\documentclass[a4paper,10pt,reqno]{amsart}
\usepackage[a4paper,tmargin=40mm,bmargin=35mm,hmargin=30mm]{geometry}
\usepackage[T1]{fontenc}
\usepackage{lmodern}
\usepackage{amsmath}
\usepackage{amssymb}
\usepackage{amsthm}
\usepackage{stmaryrd}
\usepackage{tikz}
\usetikzlibrary{cd}
\usepackage{booktabs}
\usepackage{multirow}
\usepackage{multicol}
\usepackage{enumitem}
\usepackage{here}
\usepackage{aliascnt}
\usepackage{hyperref}
\usepackage[noabbrev]{cleveref}

\newcommand{\NewTheorem}[3]{
	\newaliascnt{#1}{TheoremEnvironment}
	\newtheorem{#1}[#1]{#2}
	\aliascntresetthe{#1}
	\crefname{#1}{#2}{#3}
	\Crefname{#1}{#2}{#3}
}

\theoremstyle{definition}

\NewTheorem{definition}{Definition}{Definitions}
\NewTheorem{remark}{Remark}{Remarks}
\NewTheorem{axiom}{Axiom}{Axioms}
\NewTheorem{example}{Example}{Examples}
\NewTheorem{observation}{Observation}{Observations}
\NewTheorem{convention}{Convention}{Conventions}
\NewTheorem{notation}{Notation}{Notations}
\NewTheorem{setting}{Setting}{Settings}
\NewTheorem{question}{Question}{Questions}
\NewTheorem{answer}{Answer}{Answers}
\NewTheorem{conjecture}{Conjecture}{Conjectures}
\NewTheorem{problem}{Problem}{Problems}
\NewTheorem{solution}{Solution}{Solutions}
\NewTheorem{goal}{Goal}{Goals}
\NewTheorem{comment}{Comment}{Comments}
\NewTheorem{aim}{Aim}{Aims}
\NewTheorem{caution}{Caution}{Cautions}
\NewTheorem{exercise}{Exercise}{Exercises}

\theoremstyle{plain}
\NewTheorem{proposition}{Proposition}{Propositions}
\NewTheorem{lemma}{Lemma}{Lemmas}
\NewTheorem{theorem}{Theorem}{Theorems}
\NewTheorem{corollary}{Corollary}{Corollaries}

\crefname{enumi}{}{}
\Crefname{enumi}{}{}
\creflabelformat{enumi}{(#2#1#3)}
\crefname{enumii}{}{}
\Crefname{enumii}{}{}
\creflabelformat{enumii}{(#2#1#3)}
\crefname{enumiii}{}{}
\Crefname{enumiii}{}{}
\creflabelformat{enumiii}{(#2#1#3)}

\makeatletter
\renewcommand{\p@enumii}{}
\renewcommand{\p@enumiii}{}
\makeatother

\numberwithin{equation}{section}
\crefname{equation}{}{}
\Crefname{equation}{}{}
\creflabelformat{equation}{(#2#1#3)}

\newcommand{\SwapSymbols}[1]{
	\expandafter\let\expandafter\temporarysymbol\csname #1\endcsname
	\expandafter\let\csname #1\expandafter\endcsname\csname var#1\endcsname
	\expandafter\let\csname var#1\endcsname\temporarysymbol
}

\SwapSymbols{epsilon}
\SwapSymbols{phi}
\SwapSymbols{Gamma}
\SwapSymbols{Delta}
\SwapSymbols{Theta}
\SwapSymbols{Lambda}
\SwapSymbols{Xi}
\SwapSymbols{Pi}
\SwapSymbols{Sigma}
\SwapSymbols{Upsilon}
\SwapSymbols{Phi}
\SwapSymbols{Psi}
\SwapSymbols{Omega}

\newcommand{\bbZ}{\mathbb{Z}}

\newcommand{\cA}{\mathcal{A}}

\newcommand{\cC}{\mathcal{C}}
\newcommand{\cD}{\mathcal{D}}
\newcommand{\cE}{\mathcal{E}}

\newcommand{\cM}{\mathcal{M}}

\newcommand{\cX}{\mathcal{X}}
\newcommand{\cY}{\mathcal{Y}}
\newcommand{\cZ}{\mathcal{Z}}

\let\originalleft\left
\let\originalright\right
\renewcommand{\left}{\mathopen{}\mathclose\bgroup\originalleft}
\renewcommand{\right}{\aftergroup\egroup\originalright}

\newcommand{\isoto}{\xrightarrow{\smash{\raisebox{-0.25em}{$\sim$}}}}

\newcommand{\vin}{\rotatebox[origin=c]{90}{$\in$}}

\newcommand{\set}[2][]{\mathopen{#1\{}#2\mathclose{#1\}}}

\newcommand{\setwithtext}[2][]{\mathopen{#1\{}\,\textnormal{#2}\,\mathclose{#1\}}}
\newcommand{\setwithcondition}[3][]{\mathopen{#1\{}\,#2\mathrel{#1|}#3\,\mathclose{#1\}}}

\newcommand{\op}{\textnormal{op}}

\DeclareMathOperator{\id}{id}

\DeclareMathOperator{\Hom}{Hom}
\DeclareMathOperator{\End}{End}

\DeclareMathOperator{\Ext}{Ext}

\DeclareMathOperator{\Ker}{Ker}

\let\Im\relax
\DeclareMathOperator{\Im}{Im}

\DeclareMathOperator{\Cocont}{Cocont}

\DeclareMathOperator{\Mod}{Mod}

\newcommand{\resp}{resp.\ }

\ifpdf\else\errmessage{Use pdfLaTeX for correct typesetting}\fi

\usepackage{comment}
\usepackage{xcolor}
\usepackage{fancyvrb}

\usepackage{environ}
\NewEnviron{hide}{}
\newcommand{\hideproofs}{
	\let\proof\hide
	\let\endproof\endhide
}

\makeatletter
\def\chaptermark#1{}

\def\chapter{%
	\if@openright\cleardoublepage\else\clearpage\fi
	\thispagestyle{plain}\global\@topnum\z@
	\@afterindenttrue \secdef\@chapter\@schapter%
}
\def\@chapter[#1]#2{%
	\refstepcounter{chapter}%
	\ifnum\c@secnumdepth<\z@ \let\@secnumber\@empty
	\else \let\@secnumber\thechapter \fi
	\typeout{\chaptername\space\@secnumber}%
	\def\@toclevel{0}%
	\ifx\chaptername\appendixname \@tocwriteb\tocappendix{chapter}{#2}%
	\else \@tocwriteb\tocchapter{chapter}{#2}\fi
	\chaptermark{#1}%
	\addtocontents{lof}{\protect\addvspace{10\p@}}%
	\addtocontents{lot}{\protect\addvspace{10\p@}}%
	\@makechapterhead{#2}\@afterheading%
}
\def\@schapter#1{%
	\typeout{#1}%
	\let\@secnumber\@empty
	\def\@toclevel{0}%
	\ifx\chaptername\appendixname \@tocwriteb\tocappendix{chapter}{#1}%
	\else \@tocwriteb\tocchapter{chapter}{#1}\fi
	\chaptermark{#1}%
	\addtocontents{lof}{\protect\addvspace{10\p@}}%
	\addtocontents{lot}{\protect\addvspace{10\p@}}%
	\@makeschapterhead{#1}\@afterheading%
}
\newcommand\chaptername{Chapter}
\def\@makechapterhead#1{
	\global\topskip 7.5pc\relax
	\begingroup
	\fontsize{\@xivpt}{18}\bfseries\centering
	\ifnum\c@secnumdepth>\m@ne
	\leavevmode \hskip-\leftskip
	\rlap{\vbox to\z@{\vss
	\centerline{\normalsize\mdseries
	\uppercase\@xp{\chaptername}\enspace\thechapter}
	\vskip 3pc}}\hskip\leftskip\fi
	#1\par \endgroup
	\skip@34\p@ \advance\skip@-\normalbaselineskip
	\vskip\skip@
}
\def\@makeschapterhead#1{
	\global\topskip 7.5pc\relax
	\begingroup
	\fontsize{\@xivpt}{18}\bfseries\centering
	#1\par \endgroup
	\skip@34\p@ \advance\skip@-\normalbaselineskip
	\vskip\skip@
}
\newcounter{chapter}
\newif\if@openright
\makeatother

\title{Module-theoretic approach to dualizable Grothendieck categories}

\author{Ryo Kanda}
\address[Ryo Kanda]{Department of Mathematics, Graduate School of Science, Osaka Metropolitan University, 3-3-138, Sugimoto, Sumiyoshi, Osaka, 558-8585, Japan}
\email{ryo.kanda.math@gmail.com}

\subjclass[2020]{18E10 (Primary), 16D90, 18C35, 18E20 (Secondary)}

\keywords{Dualizability; Grothendieck category; Roos category; Gabriel-Popescu embedding; idempotent ring}

\begin{document}

\begin{abstract}
	We prove that every dualizable Grothendieck category whose dual is again a Grothendieck category satisfies Grothendieck's conditions Ab6 and Ab4*, by taking a module-theoretic approach based on the Gabriel-Popescu embedding. Combining this with a result by Stefanich, we conclude that the class of dualizable linear cocomplete categories is precisely the class of linear Grothendieck category satisfying Ab6 and Ab4*. This provides a complete answer to a modified conjecture on the dualizability, originally posed by Brandenburg, Chirvasitu, and Johnson-Freyd.
\end{abstract}

\maketitle
\tableofcontents

\section{Introduction}
\label{14208425}

In \cite{MR0190207}, Jan-Erik Roos proved the following result:

\begin{theorem}[{\cite{MR0190207}}]\label{62264625}
	Let $\cC$ be a Grothendieck category. Then the following are equivalent:
	\begin{enumerate}
		\item\label{76197744} $\cC$ satisfies Grothendieck's conditions Ab6 and Ab4*.
		\item\label{12719231} There is a pair of a ring $\Lambda$ and an idempotent ideal $I$ of $\Lambda$ such that $\cC$ is the quotient category of $\Mod\Lambda$ by the bilocalizing subcategory $\Mod(\Lambda/I)$.
	\end{enumerate}
\end{theorem}

In this paper, we refer to a Grothendieck category that satisfies Ab6 and Ab4* as a \emph{Roos category}.

Grothendieck's condition Ab6 asserts the distributive property for intersections of filtered unions (see \cref{09823809}), and Ab4* asserts that direct products are exact. The category of modules over an arbitrary ring satisfies Ab6 and Ab4*. Every locally noetherian Grothendieck category satisfies Ab6 (\cite{MR0217145}; see also \cite[Example~2.27]{MR4033823}) but does not necessarily satisfy Ab4*. So \cref{62264625} can be regarded as a characterization of Grothendieck categories that have exact direct products, under the mild assumption Ab6. \cref{62264625} has been used to investigate the (non-)exactness of direct products for various classes of Grothendieck categories, including the author's result that asserts that the category of quasi-coherent sheaves over a divisorial noetherian scheme has exact direct products if and only if $X$ is an affine scheme (\cite[Theorem~1.1]{MR4033823}).

On the other hand, the class of Grothendieck categories satisfying the condition in \cref{62264625} \cref{12719231} is fundamental in the study of (non-unital) \emph{idempotent rings}, that is, rings $A$ that do not necessarily have identities but satisfy $A^{2}=A$. For an idempotent ring $A$, the category of all $A$-modules is not really the category of interest because it is equivalent to the category of modules over the Dorroh overring of $A$, which is a unital ring (see \cite[sections 1.5 and 6.3]{MR1144522}). Instead, one of the categories of interest is the category of \emph{firm $A$-modules}, which are right $A$-modules $M$ such that the canonical homomorphism $M\otimes_{A}A\to M$ is an isomorphism. It is known that the categories arising as the category of firm modules over an idempotent ring are precisely the Grothendieck categories satisfying the condition in \cref{62264625} \cref{12719231} (see \cite[section~7]{QuillenPreprint} or \cite{MarinThesis}). Thus, \cref{62264625} can also be regarded as a characterization of the category of firm modules over an idempotent ring by a purely categorical property.

In \cite{MR3361309}, Brandenburg, Chirvasitu, and Johnson-Freyd studied the reflexivity and dualizability of locally presentable linear categories (over a field). It is known that every Grothendieck category is locally presentable. They provide several examples and non-examples of reflexive categories and dualizable categories, and in \cite[Remark~3.6]{MR3361309}, they conjectured that every dualizable locally presentable linear category is strongly generated by compact projective objects. Although there has been a result that partially confirms the conjecture (\cite{MR4366933}), there is a counterexample (\cite[Example~3.1.24]{arXiv:2307.16337}). Indeed, as a consequence of their own result in the same paper (\cite[Corollary~3.4]{MR3361309}), it follows that every Roos category is dualizable (see \cref{68013908}), and there is a nonzero Roos category that has no nonzero projective objects, constructed by Roos (\cite[Example~3]{MR0215895}). Therefore, the conjecture should be modified as follows:

\begin{conjecture}\label{20938093}
	Every dualizable locally presentable linear category is a Roos category.
\end{conjecture}

The aim of this paper is to prove the following result by taking a module-theoretic approach based on the Gabriel-Popescu embedding:

\begin{theorem}[\cref{77807299}]\label{08406530}
	Let $R$ be a commutative ring, and let $\cC$ be an $R$-linear Grothendieck category. If $\cC$ is dualizable and its dual is again a Grothendieck category, then $\cC$ is a Roos category.
\end{theorem}

Our theorem only deals with Grothendieck categories whose duals are Grothendieck categories, but in fact, Stefanich \cite[Corollary~3.1.19]{arXiv:2307.16337} has already proved that every dualizable linear \emph{cocomplete} category is a Grothendieck category satisfying Ab4*, and thus its dual is again a Grothendieck category. So, combined with our result, we conclude the following, confirming \cref{20938093}:

\begin{corollary}[\cref{04730941}]\label{93948473}
	Let $R$ be a commutative ring, and let $\cC$ be an $R$-linear cocomplete category. Then the following are equivalent:
	\begin{enumerate}
		\item $\cC$ is dualizable.
		\item $\cC$ is a Roos category.
	\end{enumerate}
\end{corollary}

\subsection*{Acknowledgments}
\label{19101576}

The author was supported by JSPS KAKENHI Grant Numbers JP21H04994 and JP24K06693, and Osaka Central Advanced Mathematical Institute: MEXT Joint Usage/Research Center on Mathematics and Theoretical Physics JPMXP0723833165.

\section{Preliminaries}
\label{09287520}

\begin{convention}\label{02947209}\leavevmode
	We use the following convention in this paper: In an abelian category, \emph{direct sums} mean coproducts, \emph{direct products} mean products. All colimits and limits are assumed to be small.
	
	A \emph{ring} means an associative ring with identity, unless otherwise specified. An associative ring that may not have identity will be referred to as a \emph{non-unital ring}. An \emph{idempotent ring} is a non-unital ring $A$ such that $A^{2}=A$.
\end{convention}

Throughout the paper, we fix a commutative ring $R$, which will be used as a base ring. An \emph{$R$-algebra} $\Lambda$ means a ring $\Lambda$ equipped with a ring homomorphism from $R$ to the center of $\Lambda$, which is equivalent to an $R$-linear category with a single object. $\Mod\Lambda$ denotes the category of right $\Lambda$-modules.

\subsection{Dualizability of locally presentable categories}
\label{63055485}

We recall the notion of dualizability and some related results; see \cite{MR3361309} for more details. \cite{MR3361309} only dealt with the case where $R$ is a field, but the results used in this paper are also valid over an arbitrary commutative ring $R$.

For locally presentable $R$-linear categories $\cC$ and $\cD$, denote by $\Cocont_{R}(\cC,\cD)$ the category of $R$-linear cocontinuous functors from $\cC$ to $\cD$, with morphisms being the natural transformations. Recall that a functor is \emph{cocontinuous} if it preserves all (small) colimits. It is known that $\Cocont_{R}(\cC,\cD)$ is again a locally presentable $R$-linear categories. If $R=\bbZ$, then $\Cocont_{\bbZ}(\cC,\cD)$ consists of all cocontinuous functors (which are necessarily additive) functors.

There is a locally presentable $R$-linear category $\cC\boxtimes_{R}\cD$ and a bifunctor $\otimes_{R}\colon\cC\times\cD\to\cC\boxtimes_{R}\cD$ that is cocontinuous and $R$-linear in each variable satisfying the following universal property: For an $R$-linear cocomplete category $\cE$, the $R$-linear functor
\begin{equation*}
	\begin{matrix}
		\Cocont_{R}(\cC\boxtimes_{R}\cD,\cE) & \to & \Cocont_{R}(\cC,\cD;\cE) \\
		\vin & & \vin\\
		F & \mapsto & F\circ\otimes_{R}
	\end{matrix}
\end{equation*}
is an equivalence, where $\Cocont_{R}(\cC,\cD;\cE)$ denotes the category of bifunctors $\cC\times\cD\to\cE$ that are cocontinuous and $R$-linear in each variable.

Define the \emph{dual} of $\cC$ to be $\cC^{*}:=\Cocont_{R}(\cC,\Mod R)$. Since $\Cocont_{R}(\cC,\Mod R)=\Cocont_{\bbZ}(\cC,\Mod\bbZ)$, the definition of $\cC^{*}$ does not depend on the choice of the base ring $R$. There is a canonical functor
\begin{equation*}
	\begin{matrix}
		\cC^{*}\boxtimes_{R}\cD & \to & \Cocont_{R}(\cC,\cD) \\
		\vin & & \vin\\
		F\otimes_{R} D & \mapsto & (C\mapsto F(C)\otimes_{R}D)
	\end{matrix}
\end{equation*}
where the functor $-\otimes_{R}D\colon\Mod R\to\cD$ is defined to be the left adjoint of $\Hom_{\cD}(D,-)\colon\cD\to\Mod R$ (the left adjoint exists because the functor preserves all limits).

\begin{definition}\label{}
	A locally presentable $R$-linear category $\cC$ is called \emph{dualizable} if the canonical functor $\cC^{*}\boxtimes_{R}\cC\to\Cocont_{R}(\cC,\cC)$ is an equivalence.
\end{definition}

Every Grothendieck category is locally presentable. For an $R$-algebra $\Lambda$, the category $\Mod\Lambda$ of right $\Lambda$-modules is an $R$-linear Grothendieck category, hence locally presentable. For $R$-algebras $\Lambda$ and $\Gamma$, the Eilenberg-Watts theorem asserts that there is an $R$-linear equivalence
\begin{equation*}
	\begin{matrix}
		\Mod(\Lambda^{\op}\otimes_{R}\Gamma) & \isoto & \Cocont_{R}(\Mod\Lambda,\Mod\Gamma) \\
		\vin & & \vin\\
		M & \mapsto & -\otimes_{\Lambda}M\rlap{.}
	\end{matrix}
\end{equation*}
Note that $\Mod(\Lambda^{\op}\otimes_{R}\Gamma)$ is canonically equivalent to the category of $\Lambda$-$\Gamma$-bimodules over $R$. In particular, $(\Mod\Lambda)^{*}=\Cocont_{R}(\Mod\Lambda,\Mod R)\cong\Mod(\Lambda^{\op})$.

It is shown in \cite[Lemma~3.5]{MR3361309} that every locally presentable $R$-linear category that admits a strongly generating set consisting of compact projective objects is dualizable. In particular, $\Mod\Lambda$ is dualizable since $\Lambda\in\Mod\Lambda$ is a compact projective object and the singleton $\set{\Lambda}$ is a strongly generating set.

\subsection{Gabriel-Popescu embedding and tensor products}
\label{14007805}

\begin{definition}\label{09238459}
	Let $\cC$ be a Grothendieck category.
	\begin{enumerate}
		\item\label{04092349} A \emph{localizing subcategory} of $\cC$ is a full subcategory that is closed under subobjects, quotient objects, extensions, and direct sums.
		\item\label{42305982} A \emph{closed subcategory} of $\cC$ is a full subcategory that is closed under subobjects, quotient objects, direct sums, and direct products.
		\item\label{02398452} A \emph{bilocalizing subcategory} of $\cC$ is a full subcategory that is both localizing and closed.
	\end{enumerate}
	\begin{table}[H]
		\begin{tabular}{cccccc}
			\toprule
			& sub & quotient & extension & direct sum & direct product \\
			\midrule
			localizing		& \checkmark & \checkmark & \checkmark & \checkmark & \\
			closed			& \checkmark & \checkmark & & \checkmark & \checkmark \\
			bilocalizing	& \checkmark & \checkmark & \checkmark & \checkmark & \checkmark \\
			\bottomrule
		\end{tabular}
	\end{table}
\end{definition}

The following result is well-known:

\begin{theorem}[{\cite[Lemma~V.2.1]{MR0232821} and \cite[Proposition~III.6.4.1]{MR1347919}}]\label{42902344}
	Let $\Lambda$ be a ring. Then there is a bijection
	\begin{equation}\label{02589345}
		\begin{matrix}
			\setwithtext{ideals of $\Lambda$} & \isoto & \setwithtext{closed subcategories of $\Mod\Lambda$} \\
			\vin & & \vin\\
			I & \mapsto & \Mod(\Lambda/I)\rlap{,}
		\end{matrix}
	\end{equation}
	where $\Mod(\Lambda/I)$ is identified with the closed subcategory consisting of all right $\Lambda$-modules $M$ with $MI=0$. For ideals $I$ and $J$ of $\Lambda$, the product $IJ$ is sent to $\Mod(\Lambda/J)*\Mod(\Lambda/I)$, which is, by definition, the full subcategory consisting of all $M\in\Mod\Lambda$ that admit a short exact sequence
	\begin{equation*}
		0\to L\to M\to N\to 0
	\end{equation*}
	with $L\in\Mod(\Lambda/J)$ and $N\in\Mod(\Lambda/I)$.
	
	As a consequence, the bijection \cref{02589345} induces a bijection
	\begin{equation*}
		\setwithtext{idempotent ideals of $\Lambda$}\isoto\setwithtext{bilocalizing subcategories of $\Mod\Lambda$}.
	\end{equation*}
	Here an ideal $I$ is called \emph{idempotent} if $I^{2}=I$.
\end{theorem}

\begin{theorem}[{Gabriel-Popescu \cite{MR0166241}; see also \cite[Theorem~4.14.2]{MR0340375}}]\label{08230321}
	Let $\cC$ be a Grothendieck category. Let $U$ be a generator in $\cC$ and set $\Lambda:=\End_{\cC}(
U)$. Then the functor
	\begin{equation*}
		S:=\Hom_{\cC}(U,-)\colon\cC\to\Mod\Lambda
	\end{equation*}
	has a left adjoint $Q\colon\Mod\Lambda\to\cC$. Let
	\begin{equation*}
		\cX:=\Ker Q=\setwithcondition{M\in\Mod\Lambda}{Q(M)=0}.
	\end{equation*}
	Then $Q$ induces an equivalence
	\begin{equation*}
		\frac{\Mod\Lambda}{\cX}\isoto\cC.
	\end{equation*}
	
	Moreover, the following hold:
	\begin{enumerate}
		\item\label{98234023} The functor $Q$ is exact and dense. $Q(\Lambda)=U$.
		\item\label{09715317} The functor $S=\Hom_{\cC}(U,-)$ is left exact and fully faithful. $S(U)=\Hom_{\cC}(U,U)=\Lambda$.
		\item\label{98727843} The essential image of $S$ consists of all $M\in\Mod\Lambda$ satisfying
		\begin{equation*}
			\Hom_{\Lambda}(X,M)=0\quad\text{and}\quad\Ext_{\Lambda}^{1}(X,M)=0\quad\text{for all $X\in\cX$}.
		\end{equation*}
		\item\label{07215821} The counit morphism $Q\circ S\to\id_{\cC}$ is an isomorphism.
	\end{enumerate}
\end{theorem}

\begin{convention}\label{87100423}
	For a Grothendieck category $\cC$, we will refer to $(\Mod\Lambda,Q,S)$ in \cref{42902344} as a \emph{Gabriel-Popescu embedding} for $\cC$. Note that the unit morphism $\eta_{\Lambda}\colon\Lambda\to SQ(\Lambda)$ is an isomorphism because $\Lambda=S(U)$ and $Q\circ S\cong\id_{\Mod\Lambda}$.
\end{convention}

\begin{theorem}[\cite{MR3873542}]\label{34641739}
	Let $\cC_{1}$ and $\cC_{2}$ be $R$-linear Grothendieck categories. For each $i=1,2$, take an $R$-algebra $\Lambda_{i}$ and a localizing subcategory $\cX_{i}$ of $\Mod\Lambda_{i}$ such that there is an $R$-linear equivalence $\cC_{i}\cong(\Mod\Lambda_{i})/\cX_{i}$. Then there is an $R$-linear equivalence
	\begin{equation*}
		\cC_{1}\boxtimes_{R}\cC_{2}\cong\dfrac{\Mod(\Lambda_{1}\otimes_{R}\Lambda_{2})}{\cX}
	\end{equation*}
	where $\cX$ is the smallest localizing subcategory of $\Mod(\Lambda_{1}\otimes_{R}\Lambda_{2})$ containing all $\Lambda_{1}\otimes_{R}\Lambda_{2}$-modules $B$ such that $B_{\Lambda_{1}}\in\cX_{1}$ or $B_{\Lambda_{2}}\in\cX_{2}$.
\end{theorem}

\begin{proof}
	The localizing subcategory $\cX_{i}$ defines a linear topology on the singleton $R$-linear category $\set{\Lambda_{i}}$ as in \cite[section~2]{MR3873542}, and the description of the tensor product follows from \cite[Theorem~5.4]{MR3873542}.
\end{proof}

\subsection{Roos categories}
\label{01314096}

\begin{definition}\label{09823809}
	Let $\cC$ be a cocomplete abelian category.
	\begin{enumerate}
		\item\label{09234092} $\cC$ is said to satisfy \emph{Ab4} if direct sums are exact, that is, for every small family of short exact sequences, its termwise direct sum is again a short exact sequence.
		\item\label{09823409} $\cC$ is said to satisfy \emph{Ab5} if filtered colimits are exact, that is, for every small filtered system of short exact sequences, its termwise filtered colimit is again a short exact sequence.
		\item\label{23978432} $\cC$ is said to satisfy \emph{Ab6} if the following condition is fulfilled: Let $M$ be an object in $\cC$ and let $\set{\set{L_{i}^{(j)}}_{i\in I_{j}}}_{j\in J}$ be a small family of (upward) filtered sets of subobjects of $M$. Then
		\begin{equation*}
			\bigcap_{j\in J}\bigg(\bigcup_{i\in I_{j}}L_{i}^{(j)}\bigg)=\bigcup_{\set{i(j)}_{j}\in\prod_{j\in J}I_{j}}\bigg(\bigcap_{j\in J}L_{i(j)}^{(j)}\bigg),
		\end{equation*}
		where the symbol $\bigcup$ means the filtered union of subobjects (that is, the sum of subobjects that are filtered).
	\end{enumerate}
	An abelian category with direct products is said to satisfy \emph{Ab4*} (\resp \emph{Ab5*}, \emph{Ab6*}) if the opposite category satisfies Ab4 (\resp Ab5, Ab6).
\end{definition}

Recall that a Grothendieck category is a cocomplete abelian category that satisfies Ab5 and has a generator. It is known that every Grothendieck category is also complete.

\begin{theorem}[{Roos \cite{MR0190207}; see also \cite[section~4.21]{MR0340375}}]\label{09243093}
	Let $(\Mod\Lambda,Q,S)$ be a Gabriel-Popescu embedding for a Grothendieck category $\cC$. Then the following are equivalent:
	\begin{enumerate}
		\item\label{02984504} $\cC$ satisfies Ab6 and Ab4*.
		\item\label{09842935} $\cX=\Ker Q$ is a bilocalizing subcategory of $\Mod\Lambda$.
		\item\label{42092340} $Q\colon\Mod\Lambda\to\cC$ has a left adjoint.
		\item\label{23092340} $Q\colon\Mod\Lambda\to\cC$ preserves direct products.
	\end{enumerate}
	If these equivalent conditions hold, then the left adjoint $H$ of $Q$ is fully faithful and thus satisfies $H\circ Q\cong\id_{\Mod\Lambda}$.
\end{theorem}

\begin{definition}\label{09874917}
	We call a Grothendieck category $\cC$ a \emph{Roos category} if it satisfies Ab6 and Ab4*.
	\begin{table}[H]
		\begin{tabular}{ccccccccc}
			\toprule
			& cocomplete & Ab4 & Ab5 & Ab6 & complete & Ab4* & Ab5* & Ab6* \\
			\midrule
			Grothendieck category	& \checkmark & \checkmark & \checkmark & & \checkmark & & & \\
			Roos category			& \checkmark & \checkmark & \checkmark & \checkmark & \checkmark & \checkmark & & \\
			\bottomrule
		\end{tabular}
	\end{table}
\end{definition}

\begin{remark}\label{23409530}
	The category of modules over an arbitrary ring is a Roos category.
	
	There exists a nonzero Roos category that admits no projective objects other than zero (\cite[Example~3]{MR0215895}; see also \cite[section~4]{MR2197371}).
\end{remark}

\begin{remark}\label{45309230}
	Let $\cC$ be a Roos category, and let $(\Mod\Lambda,Q,S)$ be a Gabriel-Popescu embedding for $\cC$. By \cref{09243093}, $\cC$ is equivalent to $(\Mod\Lambda)/\cX$, where $\cX:=\Ker Q$. Moreover, by \cref{42902344}, $\cX=\Mod(\Lambda/I)$ for some idempotent ideal $I\subset\Lambda$. Thus
	\begin{equation*}
		\cC\cong\frac{\Mod\Lambda}{\Mod(\Lambda/I)}.
	\end{equation*}
	Its dual is obtained by taking the opposite ring (\cite[Proposition~8.6]{QuillenPreprint}):
	\begin{equation}\label{41781078}
		\cC^{*}\cong\frac{\Mod\Lambda^{\op}}{\Mod(\Lambda^{\op}/I)}.
	\end{equation}
	
	Conversely, for a ring $\Lambda$ and an idempotent ideal $I\subset\Lambda$, the subcategory $\Mod(\Lambda/I)\subset\Mod\Lambda$ is bilocalizing, so the quotient category $(\Mod\Lambda)/(\Mod(\Lambda/I))$ is a Roos category (see \cite[Theorem~4.21.6]{MR0340375}).
	
	Therefore, if $\cC$ is a Roos category, its dual $\cC^{*}$ is also a Roos category (because it is of the form \cref{41781078}).
\end{remark}

\section{Main theorem}
\label{80319486}

Let $R$ be a commutative ring. The first observation in this section is that every Roos category is dualizable, and its dual is again a Roos category (hence a Grothendieck category). This follows immediate from the following result in \cite{MR3361309}:

\begin{theorem}[{\cite[Corollary~3.4]{MR3361309}}]\label{07640566}
	Let $\cM$ and $\cC$ be locally presentable $R$-linear categories and suppose that $\cM$ is dualizable. If there are cocontinuous functors $Q\colon\cM\to\cC$ and $H\colon\cC\to\cM$ such that $Q\circ H\cong\id_{\cC}$, then $\cC$ is also dualizable.
\end{theorem}

\begin{corollary}\label{68013908}
	Every $R$-linear Roos category $\cC$ is dualizable. The dual of $\cC$ is also a Roos category.
\end{corollary}

\begin{proof}
	Let $(\Mod\Lambda,Q,S)$ be a Gabriel-Popescu embedding for $\cC$. Then $Q$ has a left adjoint $H$, and $Q\circ H\cong\id_{\cC}$. Therefore, $\cC$ is dualizable by \cref{07640566}. The dual $\cC^{*}$ is also a Roos category by \cref{45309230}.
\end{proof}

Next we will show that a dualizable Grothendieck category whose dual is again a Grothendieck category is a Roos category.

\begin{lemma}\label{85148201}
	Let $\cC$ be a dualizable $R$-linear Grothendieck category whose dual is again a Grothendieck category. Let $(\Mod\Lambda,Q,S)$ be a Gabriel-Popescu embedding for $\cC$. Then there is an $R$-linear cocontinuous functor $H\colon\cC\to\Mod\Lambda$ such that $Q\circ H\cong\id_{\cC}$.
\end{lemma}

\begin{proof}
	We have the commutative diagram
	\begin{equation*}
		\begin{tikzcd}
			(\Mod\Lambda)\boxtimes_{R}\cC^{*}\ar[d,"Q\boxtimes\id_{\cC^{*}}"']\ar[r] & \Cocont_{R}(\cC,\Mod\Lambda)\ar[d,"Q_{*}"] \\
			\cC\boxtimes_{R}\cC^{*}\ar[r] & \Cocont_{R}(\cC,\cC)
		\end{tikzcd}
	\end{equation*}
	where the horizontal arrows are canonical functors and $Q_{*}$ is the induced functor defined by $Q_{*}(F):=Q\circ F$. If $Q_{*}$ is dense, then there is $H\in\Cocont_{R}(\cC,\Mod\Lambda)$ such that $Q\circ H\cong\id_{\cC}$, which is the desired claim. Since $\Mod\Lambda$ and $\cC$ are dualizable, both horizontal arrows are equivalences by \cite[Lemma~3.1]{MR3361309}. Hence, it suffices to prove that $Q\boxtimes\id_{\cC^{*}}$ is dense.
	
	Let $(\Mod\Gamma, Q', S')$ be a Gabriel-Popescu embedding for $\cC^{*}$, and let $\cX:=\Ker Q$ and $\cX':=\Ker Q'$. By \cref{34641739},
	\begin{equation*}
		\cC\boxtimes_{R}\cC^{*}\cong\dfrac{\Mod\Lambda}{\cX}\boxtimes_{R}\dfrac{\Mod\Gamma}{\cX'}\cong\dfrac{\Mod(\Lambda\otimes_{R}\Gamma)}{\cY},
	\end{equation*}
	where $\cY$ is the smallest localizing subcategories of $\Mod(\Lambda\otimes_{R}\Gamma)$ containing all $B\in\Mod(\Lambda\otimes_{R}\Gamma)$ such that $B_{\Lambda}\in\cX$ or $B_{\Gamma}\in\cX'$. Similarly, we have
	\begin{equation*}
		(\Mod\Lambda)\boxtimes_{R}\cC^{*}\cong(\Mod\Lambda)\boxtimes_{R}\dfrac{\Mod\Gamma}{\cX'}\cong\dfrac{\Mod(\Lambda\otimes_{k}\Gamma)}{\cZ},
	\end{equation*}
	where $\cZ$ is the smallest localizing subcategories of $\Mod(\Lambda\otimes_{R}\Gamma)$ containing all $B\in\Mod(\Lambda\otimes_{R}\Gamma)$ such that $B_{\Gamma}\in\cX'$. Obviously, $\cY\subset\cZ$, and we obtain the commutative diagram
	\begin{equation*}
		\begin{tikzcd}
			(\Mod\Lambda)\boxtimes_{R}\cC^{*}\ar[d,"Q\boxtimes_{R}\id_{\cC^{*}}"']\ar[r,"\sim"] & (\Mod\Lambda)\boxtimes_{R}\dfrac{\Mod\Gamma}{\cX'}\ar[d]\ar[r,"\sim"] & \dfrac{\Mod(\Lambda\otimes_{R}\Gamma)}{\cY}\ar[d]\\
			\cC\boxtimes_{R}\cC^{*}\ar[r,"\sim"] & \dfrac{\Mod\Lambda}{\cX}\boxtimes_{R}\dfrac{\Mod\Gamma}{\cX'}\ar[r,"\sim"] & \dfrac{\Mod(\Lambda\otimes_{R}\Gamma)}{\cZ}\rlap{.}
		\end{tikzcd}
	\end{equation*}
	The canonical functor represented by the right-most vertical arrow is dense, and hence $Q\boxtimes\id_{\cC^{*}}$ is also dense.
\end{proof}

\begin{lemma}\label{72974304}
	Let $\cC$ be an $R$-linear Grothendieck category, and let $(\Mod\Lambda,Q,S)$ be a Gabriel-Popescu embedding for $\cC$. If there is an $R$-linear cocontinuous functor $H\colon\cC\to\Mod\Lambda$ such that $Q\circ H\cong\id_{\cC}$, then $\cC$ is a Roos category.
\end{lemma}

\begin{proof}
	Let $\alpha\colon Q\circ H\isoto\id_{\cC}$ be an isomorphism. Since $H\circ Q\colon\Mod\Lambda\to\Mod\Lambda$ is cocontinuous, the Eilenberg-Watts theorem implies that there is a $\Lambda$-$\Lambda$-bimodule $B$ over $R$ such that $H\circ Q\cong-\otimes_{\Lambda}B$. Let $\beta\colon -\otimes_{\Lambda}B\isoto H\circ Q$ be an isomorphism, and let $\psi\colon\Lambda\otimes_{\Lambda}B\isoto B$ be the canonical bimodule isomorphism. Define the right $\Lambda$-module homomorphism $\phi\colon B\to\Lambda$ as the composition
	\begin{equation*}
		\begin{tikzcd}
			B\ar[r,"\eta_{B}"] & SQ(B) & SQ(\Lambda\otimes_{\Lambda}B)\ar[l,"SQ(\psi)"',"\sim"]\ar[r,"SQ(\beta_{\Lambda})","\sim"'] & SQHQ(\Lambda)\ar[r,"S(\alpha_{Q(\Lambda)})","\sim"'] & SQ(\Lambda) & \Lambda\ar[l,"\eta_{\Lambda}"',"\sim"]
		\end{tikzcd}
	\end{equation*}
	where $\eta\colon\id\to SQ$ is the unit morphism, and $\eta_{\Lambda}$ is an isomorphism because this is a Gabriel-Popescu embedding. To observe that $\phi$ is a bimodule homomorphism, fix an element $\lambda\in\Lambda$ and let $m_{\Lambda}\colon\Lambda\to\Lambda$ and $m_{B}\colon B\to B$ be the multiplication of $\lambda$ from the left. Then we have the commutative diagram
	\begin{equation*}
		\begin{tikzcd}
			B\ar[d,"m_{B}"]\ar[r,"\eta_{B}"] & SQ(B)\ar[d,"SQ(m_{B})"] & SQ(\Lambda\otimes_{\Lambda}B)\ar[d,"SQ(m_{\Lambda}\otimes B)"]\ar[l,"SQ(\psi)"']\ar[r,"SQ(\beta_{\Lambda})"] & SQHQ(\Lambda)\ar[d,"SQHQ(m_{\Lambda})"]\ar[r,"S(\alpha_{Q(\Lambda)})"] & SQ(\Lambda)\ar[d,"SQ(m_{\Lambda})"] & \Lambda\ar[d,"m_{\Lambda}"]\ar[l,"\eta_{\Lambda}"'] \\
			B\ar[r,"\eta_{B}"'] & SQ(B) & SQ(\Lambda\otimes_{\Lambda}B)\ar[l,"SQ(\psi)"]\ar[r,"SQ(\beta_{\Lambda})"'] & SQHQ(\Lambda)\ar[r,"S(\alpha_{Q(\Lambda)})"'] & SQ(\Lambda) & \Lambda\ar[l,"\eta_{\Lambda}"]\rlap{.}
		\end{tikzcd}
	\end{equation*}
	Hence, $\phi$ is a bimodule homomorphism. So $I:=\Im\phi$ is a (two-sided) ideal of $\Lambda$.
	
	We show that the subcategory $\cX:=\Ker Q$ is equal to $\Mod(\Lambda/I)$. Since $Q(\eta_{B})\colon Q(B)\to QSQ(B)$ is an isomorphism, $Q(\phi)\colon Q(B)\to Q(\Lambda)$ is also an isomorphism. By the exactness of $Q$, the canonical morphism $Q(I)\isoto Q(\Lambda)$ is an isomorphism, and $Q(\Lambda/I)=0$. Hence $\Lambda/I\in\cX$, and thus $\Mod(\Lambda/I)\subset\cX$.
	
	Let $M\in\cX$. Then $M\otimes_{\Lambda}B\cong HQ(M)=0$. The canonical epimorphism $B\to\Im\phi=I$ induces an epimorphism $M\otimes_{\Lambda}B\to M\otimes_{\Lambda} I$, so $M\otimes_{\Lambda}I=0$. This implies that $MI=0$ and $M\in\Mod(\Lambda/I)$. Therefore $\cX\subset\Mod(\Lambda/I)$.
	
	Consequently, the localizing subcategory $\cX$ is equal to $\Mod(\Lambda/I)$, which is also closed under direct products in $\Mod\Lambda$. Therefore $\cX$ is a bilocalizing subcategory of $\Mod\Lambda$, and $\cC\cong(\Mod\Lambda)/\cX$ is a Roos category.
\end{proof}

\begin{theorem}\label{77807299}
	For an $R$-linear Grothendieck category $\cC$, the following are equivalent:
	\begin{enumerate}
		\item\label{65302422} $\cC$ is a Roos category (Ab6 and Ab4*).
		\item\label{80445270} $\cC$ is dualizable and $\cC^{*}$ is a Grothendieck category.
		\item\label{02597871} For every/some Gabriel-Popescu embedding $(\Mod\Lambda,Q,S)$ for $\cC$, the functor $Q$ has a left adjoint.
		\item\label{67470951} For every/some Gabriel-Popescu embedding $(\Mod\Lambda,Q,S)$ for $\cC$, there is a cocontinuous functor $H\colon\cC\to\Mod\Lambda$ such that $Q\circ H\cong\id_{\cC}$.
	\end{enumerate}
\end{theorem}

\begin{proof}
	$\text{\cref{65302422}}\Rightarrow\text{\cref{80445270}}$ follows from \cref{68013908}. $\text{\cref{80445270}}\Rightarrow\text{\cref{65302422}}$ follows from \cref{85148201,72974304}.
	
	$\text{\cref{65302422}}\Leftrightarrow\text{\cref{02597871}}\Rightarrow\text{\cref{67470951}}$ is part of \cref{09243093}.
	
	$\text{\cref{67470951}}\Rightarrow\text{\cref{65302422}}$ is \cref{72974304}.
\end{proof}

\begin{remark}\label{79237878}
	Note that $\text{\cref{67470951}}\Rightarrow\text{\cref{02597871}}$ in \cref{77807299} does not hold for a cocontinuous functor $Q$ (which is not necessarily a Gabriel-Popescu quotient functor) in general. To see this, let $k$ be a field, and consider the functors
	\begin{equation*}
		Q:=-\otimes_{k[x]}(k[x]/(x))\colon\Mod k[x]\to\Mod k
	\end{equation*}
	and
	\begin{equation*}
		H:=-\otimes_{k}k[x]\colon\Mod k\to\Mod k[x].
	\end{equation*}
	Then $Q$ and $H$ are cocontinuous functors, and
	\begin{equation*}
		Q\circ H=-\otimes_{k}k[x]\otimes_{k[x]}(k[x]/(x))=-\otimes_{k}(k[x]/(x))\cong\id_{\Mod k}.
	\end{equation*}
	However, $Q$ does not have a left adjoint because $Q$ is not left exact (that is, $k[x]/(x)$ is not a flat $k[x]$-module).
\end{remark}

Stefanich \cite[Corollary~3.1.19]{arXiv:2307.16337} proved that every dualizable $R$-linear \emph{cocomplete} (not necessarily locally presentable) category $\cC$ is a Grothendieck category satisfying Ab4* (using $\cA=\Mod R$ as the symmetric monoidal Grothendieck category in the statement). Hence its dual $\cC^{*}$ is also a Grothendieck category because it is dualizable. Therefore, by \cref{77807299}, we obtain a complete characterization of dualizable linear cocomplete categories, confirming \cref{20938093}:

\begin{corollary}\label{04730941}
	Let $\cC$ be an $R$-linear cocomplete category. Then the following are equivalent:
	\begin{enumerate}
		\item $\cC$ is dualizable.
		\item $\cC$ is a Roos category.
	\end{enumerate}
\end{corollary}


\providecommand{\bysame}{\leavevmode\hbox to3em{\hrulefill}\thinspace}
\providecommand{\MR}{\relax\ifhmode\unskip\space\fi MR }
\providecommand{\MRhref}[2]{%
  \href{http://www.ams.org/mathscinet-getitem?mr=#1}{#2}
}
\providecommand{\href}[2]{#2}



\begin{thebibliography}{LRGS18}

\bibitem[BCJF15]{MR3361309}
Martin Brandenburg, Alexandru Chirvasitu, and Theo Johnson-Freyd,
  \emph{Reflexivity and dualizability in categorified linear algebra}, Theory
  Appl. Categ. \textbf{30} (2015), 808--835. \MR{3361309}

\bibitem[Chi22]{MR4366933}
Alexandru Chirvasitu, \emph{Higher dualizability and singly-generated
  {G}rothendieck categories}, Appl. Categ. Structures \textbf{30} (2022),
  no.~1, 1--12. \MR{4366933}

\bibitem[Gab62]{MR0232821}
Pierre Gabriel, \emph{Des cat{\'e}gories ab{\'e}liennes}, Bull. Soc. Math.
  France \textbf{90} (1962), 323--448. \MR{0232821}

\bibitem[Kan19]{MR4033823}
Ryo Kanda, \emph{Non-exactness of direct products of quasi-coherent sheaves},
  Doc. Math. \textbf{24} (2019), 2037--2056. \MR{4033823}

\bibitem[LRGS18]{MR3873542}
Wendy Lowen, Julia Ramos~Gonz\'{a}lez, and Boris Shoikhet, \emph{On the tensor
  product of linear sites and {G}rothendieck categories}, Int. Math. Res. Not.
  IMRN (2018), no.~21, 6698--6736. \MR{3873542}

\bibitem[Mar97]{MarinThesis}
Leandro Mar\'{\i}n, \emph{Categories of modules for idempotent rings and morita
  equivalences}, Master's thesis, University of Glasgow, 1997.

\bibitem[PG64]{MR0166241}
Nicolae Popesco and Pierre Gabriel, \emph{Caract{\'e}risation des
  cat{\'e}gories ab{\'e}liennes avec g{\'e}n{\'e}rateurs et limites inductives
  exactes}, C. R. Acad. Sci. Paris \textbf{258} (1964), 4188--4190. \MR{0166241}

\bibitem[Pop73]{MR0340375}
N.~Popescu, \emph{Abelian categories with applications to rings and modules},
  Academic Press, London, 1973, London Mathematical Society Monographs, No. 3.
  \MR{0340375}

\bibitem[Qui96]{QuillenPreprint}
Daniel Quillen, \emph{Module theory over nonunital rings}, Preprint (1996).

\bibitem[Roo65]{MR0190207}
Jan-Erik Roos, \emph{Caract{\'e}risation des cat{\'e}gories qui sont quotients
  de cat{\'e}gories de modules par des sous-cat{\'e}gories bilocalisantes}, C.
  R. Acad. Sci. Paris \textbf{261} (1965), 4954--4957. \MR{0190207}

\bibitem[Roo66]{MR0215895}
\bysame, \emph{Sur les foncteurs d{\'e}riv{\'e}s des produits infinis dans les
  cat{\'e}gories de {G}rothendieck. {E}xemples et contre-exemples}, C. R. Acad.
  Sci. Paris S{\'e}r. A-B \textbf{263} (1966), A895--A898. \MR{0215895}

\bibitem[Roo67]{MR0217145}
\bysame, \emph{Sur la condition {${\rm AB}$} {$6$} et ses variantes dans les
  cat{\'e}gories ab{\'e}liennes}, C. R. Acad. Sci. Paris S{\'e}r. A-B
  \textbf{264} (1967), A991--A994. \MR{0217145}

\bibitem[Roo06]{MR2197371}
\bysame, \emph{Derived functors of inverse limits revisited}, J. London Math.
  Soc. (2) \textbf{73} (2006), no.~1, 65--83. \MR{2197371}

\bibitem[Ros95]{MR1347919}
Alexander~L. Rosenberg, \emph{Noncommutative algebraic geometry and
  representations of quantized algebras}, Mathematics and its Applications,
  vol. 330, Kluwer Academic Publishers Group, Dordrecht, 1995. \MR{1347919}

\bibitem[{Ste}23]{arXiv:2307.16337}
Germ{\'a}n {Stefanich}, \emph{{Classification of fully dualizable linear
  categories}}, arXiv:2307.16337.

\bibitem[Wis91]{MR1144522}
Robert Wisbauer, \emph{Foundations of module and ring theory}, german ed.,
  Algebra, Logic and Applications, vol.~3, Gordon and Breach Science
  Publishers, Philadelphia, PA, 1991, A handbook for study and research.
  \MR{1144522}

\end{thebibliography}
\end{document}